\documentclass[oneside,11pt]{article}

\usepackage[utf8]{inputenc}
\usepackage[T1]{fontenc}

\usepackage{amsthm}
\usepackage{amsmath}
\usepackage{amsfonts}
\usepackage{amssymb}

\usepackage{graphicx}
\usepackage{graphbox}
\usepackage{float}
\usepackage{booktabs}
\usepackage{multirow}
\usepackage{rotating}
\usepackage{array}
\usepackage{xcolor}
\usepackage{subcaption}
\usepackage[ruled]{algorithm2e}

\newcolumntype{C}[1]{>{\centering\arraybackslash}p{#1}}
\newcolumntype{R}[1]{>{\raggedleft\let\newline\\\arraybackslash\hspace{0pt}}m{#1}}
\newcolumntype{L}[1]{>{\raggedright\let\newline\\\arraybackslash\hspace{0pt}}m{#1}}

\usepackage{breakcites}

\usepackage{nowidow}

\usepackage{enumitem}

\usepackage[top=1.5cm,bottom=2cm,right=2.5cm,left=2.5cm]{geometry}
\usepackage{titling}

\usepackage{natbib}

\usepackage{ifplatform}

\usepackage{textcomp}
\usepackage{bbm}

\usepackage{comment}

\ifwindows
\fi

\allowdisplaybreaks

\newcommand{\lb}{\left\{}
\newcommand{\rb}{\right\}}

\newcommand{\R}{\mathbb{R}}

\newcommand{\Sp}{\mathbb{S}}

\newcommand{\rd}{\mathrm{d}}

\newcommand{\bx}{\boldsymbol{x}}

\newcommand{\by}{\boldsymbol{y}}
\newcommand{\bX}{\boldsymbol{X}}
\newcommand{\bY}{\boldsymbol{Y}}

\newcommand{\bU}{\boldsymbol{U}}

\newcommand{\bmu}{\boldsymbol{\mu}}

\newcommand{\bxi}{\boldsymbol{\xi}}

\newcommand{\bphi}{\boldsymbol{\phi}}

\newcommand{\bB}{\boldsymbol{B}}

\newcommand{\bI}{\boldsymbol{I}}

\newcommand{\Prob}[1]{\mathbb{P}\lb #1\rb}

\newcommand{\defin}{:=}

\DeclareFontFamily{OT1}{pzc}{}
\DeclareFontShape{OT1}{pzc}{m}{it}{<-> s * [1.10] pzcmi7t}{}
\DeclareMathAlphabet{\mathpzc}{OT1}{pzc}{m}{it}

\newtheorem{theorem}{Theorem}[section]
\newtheorem{corollary}{Corollary}[section]

\newtheorem{proposition}{Proposition}[section]

\newcommand{\pct}{\%}

\newcommand{\Cau}{\mathrm{C}}
\newcommand{\vMF}{\mathrm{vMF}}
\newcommand{\Pois}{\mathrm{P}}
\newcommand{\Mob}{\mathrm{MvMF}}
\newcommand{\Card}{\mathrm{Card}}

\usepackage[pdftex,final,bookmarksnumbered,bookmarksopen=false,breaklinks,colorlinks]{hyperref}
\hypersetup{
	final,
	bookmarksnumbered=true,
	bookmarksopen=true,
	bookmarksopenlevel=0,
	unicode=false,
	pdftoolbar=true,
	pdfmenubar=true,
	pdffitwindow=false,
	pdftitle={Mobius transport on spheres},
	pdfdisplaydoctitle=true,
	pdftoolbar=true,
	pdfmenubar=true,
	pdflang={English},
	pdfauthor={Eduardo Garcia-Portugues, Shogo Kato},
	pdfsubject={arXiv paper},
	pdfcreator={Eduardo Garcia-Portugues},
	pdfproducer={Eduardo Garcia-Portugues},
	pdfkeywords={Directional statistics}{Mobius transformation}{Optimal transport}{Rotational symmetry}{Spherical Cauchy distribution}{Von Mises-Fisher distribution},
	pdfnewwindow=true,
	breaklinks=true,
	hidelinks,
	linkcolor=black,
	citecolor=black,
	filecolor=magenta,
	urlcolor=cyan
}

\newif\ifmain
\maintrue
\newif\ifsupplement
\supplementfalse

\newif\iffigstabs
\figstabstrue

\begin{document}

\ifmain

\title{M\"obius transport on spheres}
\setlength{\droptitle}{-1cm}
\predate{}%
\postdate{}%
\date{}

\author{Eduardo Garc\'ia-Portugu\'es$^{1,3}$ and Shogo Kato$^{2}$}
\footnotetext[1]{Department of Statistics, Universidad Carlos III de Madrid (Spain).}
\footnotetext[2]{Institute of Statistical Mathematics (Japan).}
\footnotetext[3]{Corresponding author. e-mail: \href{mailto:edgarcia@est-econ.uc3m.es}{edgarcia@est-econ.uc3m.es}.}
\maketitle

\begin{abstract}
The M\"obius transformation that generates the spherical Cauchy distribution from the uniform is the tangent--normal lift of a one-dimensional optimal transport: the monotone rearrangement of the cosine about the location axis. This identifies the probabilistic nature of the M\"obius transformation and suggests a generalization: replacing the Cauchy target by any rotationally symmetric law, for instance the Poisson kernel or spherical cardioid, yields a generalized M\"obius transformation. M\"obius transport of a von Mises--Fisher base gives tractable anisotropic distributions on the sphere, with closed-form densities that inherit the base normalizing constant and allow immediate simulation. The M\"obius--von Mises--Fisher and isotropic scaled von Mises--Fisher distributions, the latter also arising from a M\"obius transport, are illustrated on paleomagnetic directions and short-period comet orbits, where they outperform classical and recently proposed alternatives.
\end{abstract}
\begin{flushleft}
\end{flushleft}

\section{Introduction}

M\"obius transformations are a recurring device for building tractable models in directional statistics. On the circle, the wrapped Cauchy distribution arises from the uniform distribution through a M\"obius transformation, a fact underlying the wrapped Cauchy family and its inferential properties \citep{McCullagh1996,Kato2010a}, and M\"obius circular regression \citep{Downs2002}. On the torus, applying M\"obius transformations to a circular copula yields the bivariate wrapped Cauchy distribution \citep{Kato2015a}; the same device yields the M\"obius distribution on the disc \citep{Jones2004} and the spherical Cauchy distribution \citep{Kato2020}. The appeal is the same throughout: closed-form densities, exact simulation, and tractable likelihoods.

Optimal transport has recently entered directional statistics, notably the measure-transportation ranks and signs of \citet{Hallin2024}. Ours is a different, explicit use: the map generating the spherical Cauchy law from the uniform is the lift of a one-dimensional optimal transport (Section~\ref{sec:ot}). It rearranges only the cosine about the location axis and fixes the tangent direction, so mass moves along meridians in closed form.

The same lift depends on the target only through the distribution function of its cosine. Replacing the spherical Cauchy by any rotationally symmetric distribution therefore defines a \emph{generalized} M\"obius transformation, with the Cauchy case as the only conformal map in the family (Section~\ref{sec:transport}). Transport through it inherits the base normalizing constant, so a von Mises--Fisher base gives tractable anisotropic models (Section~\ref{sec:tvmf}). Such models are scarce: the standard anisotropic distribution of \citet{Kent1982} has a normalizing constant given by an infinite series. The M\"obius--von Mises--Fisher and isotropic scaled von Mises--Fisher distributions avoid this. We illustrate the usefulness of these transported von Mises--Fisher distributions on paleomagnetic directions (Section~\ref{sec:paleo}) and short-period comet orbits (Section~\ref{sec:comets}).

\section{The M\"obius transformation as optimal transport}
\label{sec:ot}

Let $\Sp^d\defin \{\bx\in\R^{d+1}:\|\bx\|=1\}$ be the $d$-dimensional hypersphere, $d\geq 1$, with surface area $\omega_d = 2\pi^{(d+1)/2}/\Gamma\{(d+1)/2\}$. The spherical Cauchy distribution of \citet{Kato2020}, denoted $\Cau(\bmu,\rho)$, has density
\begin{align*}
    f_{\Cau}(\bx;\bmu,\rho)\defin\frac{1}{\omega_d}\left(\frac{1-\rho^2}{1+\rho^2-2\rho \bx^\top\bmu}\right)^d,
\end{align*}
with respect to the surface measure $\sigma_d$ on $\Sp^d$, for a location $\bmu\in\Sp^d$ and concentration $\rho\in(-1,1)$; here $\rho=0$ gives the uniform distribution $\mathrm{Unif}(\Sp^d)$ and $\rho<0$ gives $\Cau(-\bmu,|\rho|)$. For $\bX\sim\Cau(\bmu,\rho)$, the projection $\bX^\top\bmu$ has the cumulative distribution function (cdf) given in the following result.

\begin{proposition}\label{prop:cdf}
Let $s\in[-1,1]$ and
\begin{align*}
M_\rho(s)\defin\frac{(1+\rho^2)s-2\rho}{1+\rho^2-2\rho s},\qquad \rho\in(-1,1),
\end{align*}
which is an increasing bijection of $[-1,1]$ with inverse $M_\rho^{-1}=M_{-\rho}$. Then, for $x \in [-1,1]$,
\begin{align}
  F_{\rho}(x) \defin \Prob{\bX^\top\bmu\leq x} = \mathrm{I}_{(1+M_\rho(x))/2}(d/2, d/2),\label{eq:cdf}
\end{align}
where $\mathrm{I}_z(a,b) \defin \mathrm{B}(z;a,b)/\mathrm{B}(a,b)$ is the regularized incomplete beta function.
\end{proposition}

For $\rho=0$, $M_0=\mathrm{id}$, so $F_{0}(x)=\mathrm{I}_{(1+x)/2}(d/2,d/2)$ is the projected uniform cdf; in particular $F_0(x)=(x+1)/2$ for $d=2$. An important consequence is that \eqref{eq:cdf} is expressible as $F_{\rho}=F_0\circ M_\rho$.

Let $\nu_\rho$ denote the law of $\bX^\top\bmu$ under $\Cau(\bmu,\rho)$; $\nu_0$ is then the projected uniform law.

\begin{proposition}\label{prop:ot}
The map $M_\rho$ equals the monotone rearrangement
\begin{align}
  M_\rho = F_0^{-1}\circ F_\rho,\label{eq:rearrangement}
\end{align}
which pushes $\nu_\rho$ forward to $\nu_0$, hence is the associated optimal transport map for every cost $c(s,t)=h(s-t)$ with $h$ strictly convex (Brenier's map if $h(s)=s^2$). Its inverse $M_{-\rho}=F_\rho^{-1}\circ F_0$ transports $\nu_0$~to~$\nu_\rho$.
\end{proposition}

This transport lifts to $\Sp^d$ through the tangent--normal decomposition. Write $\bx = t\bmu + (1-t^2)^{1/2}\,\bB_{\bmu}\bxi$ with $t = \bx^\top\bmu$, $\bxi \in \Sp^{d-1}$, and $\bB_{\bmu}$ a $(d+1)\times d$ semi-orthogonal matrix satisfying $\bB_{\bmu}\bB_{\bmu}^\top = \bI_{d+1} - \bmu\bmu^\top$ and $\bB_{\bmu}^\top\bB_{\bmu} = \bI_d$. Lifting the generating map $M_{-\rho}=F_\rho^{-1}\circ F_0$ defines $\Phi_\rho:\Sp^d\to\Sp^d$,
\begin{align*}
  \Phi_\rho\big(t\bmu + (1-t^2)^{1/2}\,\bB_{\bmu}\bxi\big) \defin (F_\rho^{-1}\circ F_0)(t)\,\bmu + \{1-(F_\rho^{-1}\circ F_0)(t)^2\}^{1/2}\,\bB_{\bmu}\bxi,
\end{align*}
which moves only the cosine and fixes the tangent direction $\bxi$, with $\Phi_\rho(\pm\bmu)\defin\pm\bmu$.

As shown next, $\Phi_\rho$ is the M\"obius transformation on the sphere of \citet{Kato2020}.

\begin{corollary}\label{cor:lift}
If $\bU\sim\mathrm{Unif}(\Sp^d)$, then $\Phi_\rho(\bU)\sim\Cau(\bmu,\rho)$. Moreover, in stereographic coordinates from $-\bmu$, $\Phi_\rho$ is the dilation $\by\mapsto \{(1-\rho)/(1+\rho)\}\by$, hence the M\"obius transformation.
\end{corollary}

\section{Generalizing the M\"obius transformation}
\label{sec:transport}

The lift $\Phi_\rho$ depends on the target only through its cosine cdf $F_\rho$; replacing $F_\rho$ therefore generalizes it. Let $G$ be rotationally symmetric about $\bmu\in\Sp^d$ with continuous, strictly increasing projected cdf $F_G$, and define the \emph{generalized M\"obius transformation} $\mathcal{M}_{G}:\Sp^d\to\Sp^d$ by
\begin{align*}
  \mathcal{M}_{G}\big(t\bmu+(1-t^2)^{1/2}\,\bB_{\bmu}\bxi\big)\defin
  (F_G^{-1}\circ F_0)(t)\,\bmu+[1-\{(F_G^{-1}\circ F_0)(t)\}^2]^{1/2}\,\bB_{\bmu}\bxi.
\end{align*}
Like $\Phi_\rho$, it rearranges the cosine by $F_G^{-1}\circ F_0$ and fixes the tangent direction; by construction $\mathcal{M}_{G}(\bU)\sim G$ for $\bU\sim\mathrm{Unif}(\Sp^d)$, so it is the lift of the one-dimensional optimal transport carrying the uniform to $G$; we call this construction \emph{M\"obius transport}.

The classical M\"obius transformation is the particular case with spherical Cauchy target: $\mathcal{M}_{\Cau(\bmu,\rho)}=\Phi_\rho$, which for $\bphi=\rho\bmu$ in the open unit ball $\mathbb{B}^{d+1}\defin\{\bphi\in\R^{d+1}:\|\bphi\|<1\}$ is the conformal map
\begin{align*}
  M_{\bphi}(\bx)\defin\frac{1-\|\bphi\|^2}{\|\bx+\bphi\|^2}\,(\bx+\bphi)+\bphi,\qquad M_{\bphi}^{-1}=M_{-\bphi},
\end{align*}
of \citet{Kato2020}. For $d\geq2$, the spherical Cauchy target is the only one whose transformation $\mathcal{M}_G$ is conformal. Angle preservation forces equal stretching along and across meridians; for the cosine rearrangement $\psi=F_G^{-1}\circ F_0$, this requirement reads $\psi'(t)=\{1-\psi(t)^2\}/(1-t^2)$, whose only increasing solutions are the M\"obius maps $\psi=M_{-\rho}$, $\rho\in(-1,1)$. On $\Sp^2$ the inverse rearrangement is elementary, $F_0^{-1}(u)=2u-1$, so $\mathcal{M}_G^{-1}$ acts on the cosine $s=\by^\top\bmu$ by $s\mapsto 2F_G(s)-1$.

Transporting a base distribution through $\mathcal{M}_G$ generates new distributions with explicit densities.

\begin{theorem}\label{thm:transport}
Let $\bX$ have density $f$ on $\Sp^d$ and let $\bY=\mathcal{M}_G(\bX)$, where the target $G$ is rotationally symmetric about $\bmu\in\Sp^d$ with density $g_G(\by)=h_G(\by^\top\bmu)/\omega_d$ and continuous, strictly increasing projected cdf. Then $\bY$ has density
\begin{align}
  f_{\bY}(\by)=f\big(\mathcal{M}_G^{-1}(\by)\big)\,h_G(\by^\top\bmu).\label{eq:pushforward}
\end{align}
\end{theorem}

A non-uniform base with location axis $\bmu_1$ acquires a second axis through the warped cosine $\bmu_1^\top\mathcal{M}_G^{-1}(\by)$, rendering the transport anisotropic and increasingly skewed as the target concentrates (Figure~\ref{fig:shapes}).

Compositions with different axes generate further anisotropy. If both targets are spherical Cauchy, $\mathcal{M}_{G_2}\circ\mathcal{M}_{G_1}$ is a rotation composed with some $M_{\bphi}$, so $\mathcal{M}_{G_2}\{\mathcal{M}_{G_1}(\bU)\}$ remains spherical Cauchy \citep{Kato2020}. For general targets $G_1$ and $G_2$ about $\bmu_1$ and $\bmu_2$, $\mathcal{M}_{G_2}\{\mathcal{M}_{G_1}(\bU)\}$ is distributed as $\mathcal{M}_{G_2}(\bX)$ with $\bX\sim G_1$, and its density \eqref{eq:pushforward} is skewed or bimodal when $\bmu_1\neq\pm\bmu_2$; the distributions of Section~\ref{sec:tvmf} are exactly of this form, a von Mises--Fisher base being itself a M\"obius transport of the uniform. Iterating \eqref{eq:pushforward} for $\mathcal{M}_{G_p}\circ\cdots\circ\mathcal{M}_{G_1}(\bU)$ gives explicit multi-axis densities, not pursued here.

\begin{figure}[h!]
\centering
\begin{subfigure}[t]{0.31\textwidth}
\centering
\includegraphics[width=\linewidth]{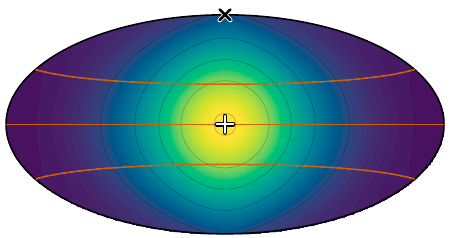}
\caption{$\lambda=0$.}
\end{subfigure}\hfill
\begin{subfigure}[t]{0.31\textwidth}
\centering
\includegraphics[width=\linewidth]{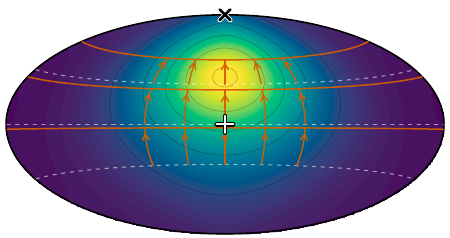}
\caption{$\lambda=1$.}
\end{subfigure}\hfill
\begin{subfigure}[t]{0.31\textwidth}
\centering
\includegraphics[width=\linewidth]{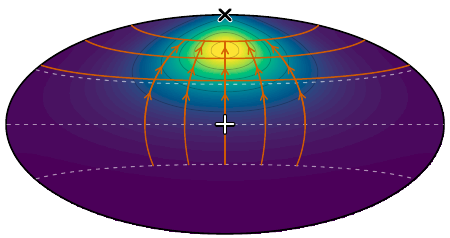}
\caption{$\lambda=3$.}
\end{subfigure}
\caption{\small Densities \eqref{eq:genvmf} on $\Sp^2$ generated from a $\vMF(\bmu_1,3)$ base (see Section~\ref{sec:mvmf}) by the generalized M\"obius transformation with target $G=\vMF(\bmu_2,\lambda)$. Hammer projections are centred on $\bmu_1$ ($+$). White dashed curves are source cosine levels about $\bmu_2$ ($\times$); orange curves and arrows show their transported images. Each panel has its own square-root colour scale.}
\label{fig:shapes}
\end{figure}

A dual mechanism for spherical asymmetry leaves the distribution of the cosine $\by^\top\bmu$ unchanged and instead perturbs the tangent sign vector: the skew-rotationally-symmetric distributions of \citet{Ley2017b} multiply a rotationally symmetric base by a skewing function of the tangent projection, preserving the base normalizing constant, as does M\"obius transport. The two mechanisms are dual and coincide only under rotational symmetry. The \emph{skew von Mises--Fisher} model instance, a von Mises--Fisher base with Gaussian skewing, is considered as a further competitor in Section~\ref{sec:paleo}.

\subsection{Some closed-form transformations}
\label{sec:closedform}

In addition to the spherical Cauchy, two further rotationally symmetric targets have projected cdfs in closed form and therefore explicit mass rearrangements. The \emph{Poisson kernel distribution} \citep[see][]{Golzy2020,McCullagh1989} has density
\begin{align*}
  f_{\Pois}(\bx;\bmu,r)=\frac{1}{\omega_d}\,\frac{1-r^2}{(1+r^2-2r\,\bx^\top\bmu)^{(d+1)/2}},\qquad r\in[0,1),
\end{align*}
the Poisson kernel of the ball $\mathbb{B}^{d+1}$ at $r\bmu$; it is the wrapped Cauchy at $d=1$ but, having exponent $(d+1)/2$ rather than $d$, is distinct from the spherical Cauchy for $d\geq2$. Its projection is the univariate family of \citet{McCullagh1989}. The \emph{spherical cardioid of order $k$} \citep{Garcia-Portugues:cardioid} has density
\begin{align*}
  f_{\Card_k}(\bx;\bmu,\gamma)=\frac{1}{\omega_d}\big\{1+\gamma\,\tilde{C}_k^{(d-1)/2}(\bx^\top\bmu)\big\},
  \qquad |\gamma|\leq 1,\ k\geq1,
\end{align*}
with $\tilde{C}_k^{\alpha}=C_k^{\alpha}/C_k^{\alpha}(1)$ the normalized Gegenbauer polynomial for $\alpha>0$ and $\tilde C_k^0=T_k$ the Chebyshev polynomial. For $d=1$ and $k=1$, it reduces to the classical circular cardioid. Their cosine rearrangements are determined by the projected cdfs and quantile functions below.

\begin{proposition}\label{prop:proj}
For $d=2$, the Poisson kernel and order-$k$ cardioid projections have the explicit~cdfs
\begin{align}
  F_{\Pois}(x) &= \frac{1-r^2}{2r}\bigg\{\frac{1}{(1+r^2-2rx)^{1/2}}-\frac{1}{1+r}\bigg\},\qquad 0<r<1,\label{eq:projP}\\
  F_{\Card_k}(x) &= \frac{x+1}{2}+\frac{\gamma}{2(2k+1)}\big\{P_{k+1}(x)-P_{k-1}(x)\big\},\label{eq:projcard}
\end{align}
where the $r=0$ limit of \eqref{eq:projP} is $(x+1)/2$ and
$P_k=\tilde C_k^{1/2}$ denotes the Legendre polynomial. In particular, $F_{\Card_1}(x)=(x+1)/2-\gamma(1-x^2)/4$ and $F_{\Card_2}(x)=\{\gamma x^3+(2-\gamma)x+2\}/4$.
\end{proposition}

\begin{corollary}\label{cor:quantile}
For $d=2$ and $u,\gamma,r\in(0,1)$, set $q\defin2(1-2u)/\gamma$, $p\defin(2-\gamma)/\gamma$, and $\Delta\defin q^2/4+p^3/27>0$. Then
\begin{align}
  F_{\Pois}^{-1}(u)&=\{1+r^2-A(u)^{-2}\}/(2r),\qquad A(u)=2ru/(1-r^2)+(1+r)^{-1},\nonumber\\
  F_{\Card_1}^{-1}(u)&=[\{(1-\gamma)^2+4\gamma u\}^{1/2}-1]/\gamma,\qquad F_{\Card_2}^{-1}(u)=\left\{-q/2+\Delta^{1/2}\right\}^{1/3}-\left\{q/2+\Delta^{1/2}\right\}^{1/3}.\label{eq:Q2}
\end{align}
\end{corollary}

\section{Transported von Mises--Fisher distributions}
\label{sec:tvmf}

\subsection{M\"obius--von Mises--Fisher}
\label{sec:mvmf}

Take the von Mises--Fisher $\vMF(\bmu_1,\kappa)$, with density $\bx\mapsto c_d^{\vMF}(\kappa)\exp(\kappa\,\bmu_1^\top\bx)$ and constant $c_d^{\vMF}(\kappa)\defin\kappa^{(d-1)/2}/\{(2\pi)^{(d+1)/2}I_{(d-1)/2}(\kappa)\}$ ($I_\nu$ the modified Bessel function of the first kind) for $\kappa>0$, with $c_d^{\vMF}(0)=1/\omega_d$. By Theorem~\ref{thm:transport}, $\bY=\mathcal{M}_G(\bX)$ with $\bX\sim\vMF(\bmu_1,\kappa)$ and a target $G$ about $\bmu_2$ has density
\begin{align}
  f_{\bY}(\by)=c_d^{\vMF}(\kappa)\,h_G(\by^\top\bmu_2)\,\exp\!\big\{\kappa\,\bmu_1^\top\mathcal{M}_G^{-1}(\by)\big\}.\label{eq:genvmf}
\end{align}
On $\Sp^2$, $\bmu_1^\top\mathcal{M}_G^{-1}(\by)=t\,\bmu_1^\top\bmu_2+\{(1-t^2)/(1-s^2)\}^{1/2}\,\bmu_1^\top(\by-s\bmu_2)$ with $s=\by^\top\bmu_2$ and $t=2F_G(s)-1$, so \eqref{eq:genvmf} is in closed form whenever $F_G$ is. At $s=\pm1$ the expression is interpreted continuously as $t\,\bmu_1^\top\bmu_2$.

The spherical Cauchy target yields the flagship member: with $\bphi=\rho\bmu_2$ one has $\mathcal{M}_G^{-1}=M_{-\bphi}$ and $h_G$ the Cauchy factor, giving the \emph{M\"obius--von Mises--Fisher} distribution $\Mob(\bmu_1,\kappa,\bphi)$, the law of $\bY=M_{\bphi}(\bX)$, with density
\begin{align}
  f_{\bY}(\by)=c_d^{\vMF}(\kappa)\exp\!\big\{\kappa\,\bmu_1^\top M_{-\bphi}(\by)\big\}
  \left(\frac{1-\rho^2}{1+\rho^2-2\rho\,\by^\top\bmu_2}\right)^{\!d}.\label{eq:mvmf}
\end{align}
This distribution interpolates familiar models: $\rho=0$ gives $\vMF(\bmu_1,\kappa)$ and $\kappa=0$ gives $\Cau(\bmu_2,\rho)$; $\bmu_2=\pm\bmu_1$ gives a rotationally symmetric law. It is anisotropic when $\kappa\rho\neq0$ and $\bmu_1\neq \pm\bmu_2$. Simulation is immediate, by drawing $\bX\sim\vMF(\bmu_1,\kappa)$ and setting $\bY=M_{\bphi}(\bX)$, and the log-likelihood is in closed form. For $d=1$, \eqref{eq:mvmf} reduces to the density of \citet{Kato2010a}. On $\Sp^2$ it has six parameters against five for Kent, whose normalizing constant, unlike \eqref{eq:mvmf}, is not fully~explicit.

Other targets are equally tractable. On $\Sp^2$, the order-$k$ cardioid target gives \eqref{eq:genvmf} with $h_G(s)=1+\gamma P_k(s)$ and the Poisson kernel target with $h_G(s)=(1-r^2)/(1+r^2-2rs)^{3/2}$, in each case $t=2F_G(s)-1$ with $F_G$ from Proposition~\ref{prop:proj}. Both reduce to $\vMF(\bmu_1,\kappa)$ when the target becomes uniform ($\gamma=0$ or $r=0$) and to the target about $\bmu_2$ when $\kappa=0$.

\subsection{Isotropic scaled von Mises--Fisher}
\label{sec:isvmf}

The \emph{scaled von Mises--Fisher} distribution of \citet{Scealy2019} is the law of $\boldsymbol{A}\bX/\|\boldsymbol{A}\bX\|$ for $\bX\sim\vMF(\bmu_1,\kappa)$ and a positive-definite scaling $\boldsymbol{A}$. Under isotropic scaling about $\bmu_2$, this normalized scaling is a cosine rearrangement; its uniform image is the \emph{angular central Gaussian} target of \mbox{\citet{Tyler1987}}.

\begin{proposition}\label{prop:svmf}
Let $\boldsymbol{A}=\bmu_2\bmu_2^\top+a(\bI_{d+1}-\bmu_2\bmu_2^\top)$ and $S_a(\bx)=\boldsymbol{A}\bx/\|\boldsymbol{A}\bx\|$ for $a>0$, and let $G_a$ be the law of $S_a(\bU)$ for $\bU\sim\mathrm{Unif}(\Sp^d)$. Then $G_a$ is rotationally symmetric about $\bmu_2$, with density
\begin{align}
  g_{G_a}(\bx)=\frac{a}{\omega_d}\,\big\{1-(1-a^2)(\bx^\top\bmu_2)^2\big\}^{-(d+1)/2}.\label{eq:acg}
\end{align}
Its cosine rearrangement $\tau_a\defin F_{G_a}^{-1}\circ F_0$ is
$\tau_a(t)=t/\{a^2+(1-a^2)t^2\}^{1/2}$, with $\tau_a^{-1}=\tau_{1/a}$, and $\mathcal{M}_{G_a}=S_a$. If $\bmu_1=\bmu_2=\bmu$, $\bX\sim\vMF(\bmu,\kappa)$ and $\bY=S_a(\bX)$, then $\bY$ has density
\begin{align}
  \by\mapsto c_d^{\vMF}(\kappa)\,a\,\big\{1-(1-a^2)(\by^\top\bmu)^2\big\}^{-(d+1)/2}
  \exp\!\bigg\{\frac{\kappa\,a\,\by^\top\bmu}{\{1-(1-a^2)(\by^\top\bmu)^2\}^{1/2}}\bigg\}.\label{eq:svmf}
\end{align}
\end{proposition}

The rearrangement $\tau_a$ draws the cosine towards the poles for $a<1$ and towards the equator for $a>1$, the scaling counterpart of the M\"obius dilation $M_{-\rho}$. At $a=1$, \eqref{eq:svmf} reduces to $\vMF(\bmu,\kappa)$. The \emph{isotropic scaled von Mises--Fisher} distribution thus belongs to the transported von Mises--Fisher family, with $\bmu_1\neq\bmu_2$ giving an anisotropic two-axis variant whose density is \eqref{eq:genvmf} with $h_{G_a}(s)=a\{1-(1-a^2)s^2\}^{-(d+1)/2}$ and $\mathcal{M}_{G_a}^{-1}=S_{1/a}$. An anisotropic scaling with unequal eigenvalues instead mixes the tangent directions, so the general elliptical form of \citet{Scealy2019} is not a cosine~rearrangement.

\section{Applications}
\label{sec:applications}

\subsection{Paleomagnetic directions}
\label{sec:paleo}

Paleomagnetic directions are a classical anisotropic testbed on $\Sp^2$, where departures from rotational symmetry are of intrinsic interest \citep{Scealy2019}. We use the $n=239$ site-mean characteristic remanence directions of the Paleocene South Ardo section (Belluno Basin, southern Alps; \citealp{Dallanave2012}), retrieved through the R package \texttt{PmagDiR} \citep{Dallanave2024}. These directions have two nearly antipodal polarities recording geomagnetic reversals; we bring them to a common polarity by reversing the directions lying more than $90^\circ$ from the maximum-variance axis, using \texttt{PmagDiR::common\_DI()}, yielding $82$ normal and $157$ reversed directions whose modal means lie $3.2^\circ$ apart after reflection. We fit seven distributions by maximum likelihood: von Mises--Fisher, skew von Mises--Fisher \citep{Ley2017b}, scaled von Mises--Fisher in general and isotropic forms, tangent von Mises--Fisher \citep{Garcia-Portugues2020}, M\"obius--von Mises--Fisher, and Kent (Table~\ref{tab:fit}).

\begin{table}[h!]
\caption{\small Maximum likelihood fits to the $n=239$ paleomagnetic directions. The columns show the number of parameters $p$, maximized log-likelihood $\ell$, and AIC/BIC.}
\label{tab:fit}
\centering
\small
\begin{tabular}{lrrrr}
\toprule
Model & $p$ & $\ell$ & AIC & BIC\\
\midrule
Von Mises--Fisher & $3$ & $41.7$ & $-77.5$ & $-67.0$\\
Skew von Mises--Fisher & $5$ & $52.9$ & $-95.9$ & $-78.5$\\
Scaled von Mises--Fisher & $6$ & $61.3$ & $-110.6$ & $-89.7$\\
Isotropic scaled von Mises--Fisher & $4$ & $58.7$ & $-109.5$ & $-95.6$\\
Tangent von Mises--Fisher & $5$ & $57.9$ & $-105.8$ & $-88.5$\\
M\"obius--von Mises--Fisher & $6$ & $63.1$ & $-114.2$ & $-93.3$\\
Kent & $5$ & $45.8$ & $-81.5$ & $-64.1$\\
\bottomrule
\end{tabular}
\end{table}

The pooled cluster is skewed rather than merely elliptical, its dispersion drawn towards shallow inclinations. The von Mises--Fisher, Kent (ovalness $\hat\beta=2.064$), and the scaled and tangent von Mises--Fisher models capture concentration and elongation but not this asymmetry; even the skew von Mises--Fisher fits distinctly worse (Table~\ref{tab:fit}). The M\"obius--von Mises--Fisher captures the asymmetry, attaining the highest likelihood and best AIC ($\hat\rho=0.51$, $\hat\kappa=5.55$, axes $72^\circ$ apart; Figure~\ref{fig:paleo}). Only the parsimonious isotropic scaled von Mises--Fisher, itself a transport member (Proposition~\ref{prop:svmf}), is preferred by BIC.

\begin{figure}[h!]
\centering
\begin{subfigure}[t]{0.37\textwidth}
\centering
\includegraphics[width=\linewidth]{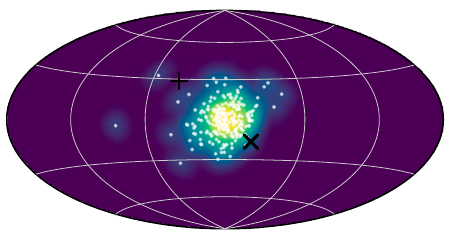}
\caption{Kernel density estimate.}
\label{fig:paleo-kde}
\end{subfigure}\hfill
\begin{subfigure}[t]{0.10\textwidth}
\centering
\includegraphics[height=92.99pt,trim=1.9bp 0 14.2bp 0,clip]{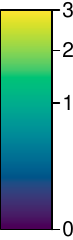}%
\end{subfigure}\hfill
\begin{subfigure}[t]{0.37\textwidth}
\centering
\includegraphics[width=\linewidth]{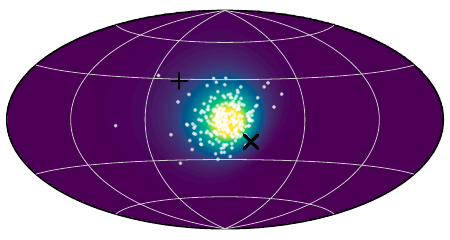}
\caption{Fitted density \eqref{eq:mvmf}.}
\label{fig:paleo-mvmf}
\end{subfigure}
\caption{\small Paleomagnetic directions (white points) in the Hammer projection centred on the sample mean, with fitted location axis $\hat\bmu_1$ ($+$) and skewness axis $\hat\bmu_2$ ($\times$). The kernel density estimate uses a rule-of-thumb bandwidth \citep{Garcia-Portugues2013a}. The panels share one square-root density scale and show the same skewed concentration.}
\label{fig:paleo}
\end{figure}

\subsection{Short-period comet orbits}
\label{sec:comets}

We analyse the $n=784$ non-fragment short-period comets (period below $200$ years) of \texttt{sphunif} \citep{Garcia-Portugues2020c}. Each orbit normal is $(\sin i\sin\Omega,-\sin i\cos\Omega,\cos i)^\top$, with inclination $i$ from the ecliptic normal and $\Omega$ the longitude of the ascending node. Unlike long-period comet orbit normals, close to uniform and well described by a spherical cardioid distribution \citep{Garcia-Portugues:cardioid}, these mostly Jupiter-family orbits are heavily clustered about the ecliptic normal and not rejected as rotationally symmetric \citep{Garcia-Portugues:Sobolev}. We fit four rotationally symmetric models: von Mises--Fisher, spherical Cauchy, Poisson kernel, and isotropic scaled von Mises--Fisher (Table~\ref{tab:comets}).

\pagebreak

The isotropic scaled von Mises--Fisher distribution fits the strongly concentrated, heavy-tailed normals best, by a wide margin in AIC and BIC, followed by the Poisson kernel, spherical Cauchy, and von Mises--Fisher models; its projected cdf is also the closest to the empirical one (Figure~\ref{fig:comets}).

\begin{table}[h!]
\caption{\small Maximum likelihood fits to the orbit normals of $n=784$ non-fragment short-period comets. Columns are as in Table~\ref{tab:fit}.}
\label{tab:comets}
\centering
\small
\begin{tabular}{lrrrr}
\toprule
Model & $p$ & $\ell$ & AIC & BIC\\
\midrule
Von Mises--Fisher & $3$ & $-759.0$ & $1524.0$ & $1538.0$\\
Spherical Cauchy & $3$ & $-210.0$ & $426.1$ & $440.1$\\
Poisson kernel & $3$ & $-203.0$ & $412.0$ & $426.0$\\
Isotropic scaled von Mises--Fisher & $4$ & $-162.3$ & $332.6$ & $351.2$\\
\bottomrule
\end{tabular}
\end{table}

\begin{figure}[h!]
\centering
\includegraphics[width=0.75\textwidth]{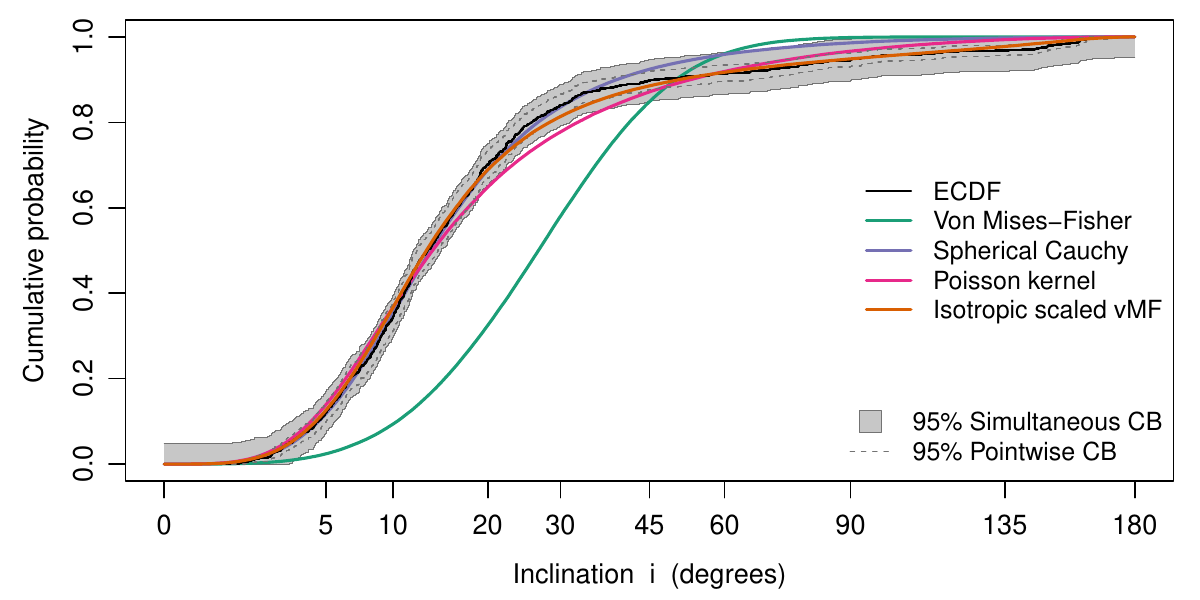}
\caption{\small Short-period comet orbit normals: empirical projected cdf and projected cdfs of the four fitted models, expressed in terms of the inclination $i$ and shown on a square-root axis. Shading is a simultaneous $95\pct$ Kolmogorov--Smirnov band; dashed grey curves are pointwise $95\pct$ bands.}
\label{fig:comets}
\end{figure}

\section*{Acknowledgments}

Financial support from grant PID2024-158399NB-I00, funded by MICIU/AEI/10.13039/501100011033 and ERDF/EU, is acknowledged.

\appendix

\section{Proofs}

\begin{proof}[Proof of Proposition~\ref{prop:cdf}]
The tangent--normal factorization gives the cosine density proportional to $s\mapsto \{(1-\rho^2)/(1+\rho^2-2\rho s)\}^d(1-s^2)^{d/2-1}$. The substitution $t=M_\rho(s)$ satisfies $\rd t=M_\rho'(s)\,\rd s$ with $M_\rho'(s)=\{(1-\rho^2)/(1+\rho^2-2\rho s)\}^2>0$, and $1-t^2=(1-s^2)M_\rho'(s)$; since $M_\rho(\pm1)=\pm1$, $M_\rho$ is an increasing bijection of $[-1,1]$, and $M_\rho^{-1}=M_{-\rho}$ by direct computation. The substitution turns the cosine density into the projected uniform density proportional to $t\mapsto (1-t^2)^{d/2-1}$, so $F_\rho(x)=F_0(M_\rho(x))$, and $F_0(y)=\mathrm{I}_{(1+y)/2}(d/2,d/2)$ gives \eqref{eq:cdf}.
\end{proof}

\begin{proof}[Proof of Proposition~\ref{prop:ot}]
By Proposition~\ref{prop:cdf}, $F_\rho=F_0\circ M_\rho$ with $F_0:[-1,1]\to[0,1]$ a continuous increasing bijection; composing with $F_0^{-1}$ gives \eqref{eq:rearrangement}. An increasing map between atomless laws on $\R$ is their monotone rearrangement, which for a cost $c(s,t)=h(s-t)$ with $h$ strictly convex is the unique optimal transport map \citep[Theorem~2.9]{Santambrogio2015}; for $h(s)=s^2$ this is the map of \citet{Brenier1991}. Interchanging the roles of $\nu_\rho$ and $\nu_0$ shows that $F_\rho^{-1}\circ F_0$, which equals $M_{-\rho}$ by \eqref{eq:rearrangement} and $M_\rho^{-1}=M_{-\rho}$, is the monotone rearrangement transporting $\nu_0$ to $\nu_\rho$.
\end{proof}

\begin{proof}[Proof of Corollary~\ref{cor:lift}]
Write the $\mathrm{Unif}(\Sp^d)$ and $\Cau(\bmu,\rho)$ laws in the tangent--normal decomposition: a law rotationally symmetric about $\bmu$ factorizes as its cosine marginal times an independent $\bxi\sim\mathrm{Unif}(\Sp^{d-1})$. Since $\Phi_\rho$ fixes $\bxi$ and maps the cosine by $F_{\rho}^{-1}\circ F_0$, it carries one factorization to the other, hence $\Phi_\rho(\bU)\sim\Cau(\bmu,\rho)$. For the second claim, stereographic projection from $-\bmu$ sends $t\bmu + (1-t^2)^{1/2}\,\bB_{\bmu}\bxi$ to $\by = r\bxi$ with $r = \{(1-t)/(1+t)\}^{1/2}$, preserving $\bxi$; the dilation $\by\mapsto c\by$ acts on the cosine through $(1-t')/(1+t') = c^2(1-t)/(1+t)$, and $c = (1-\rho)/(1+\rho)$ yields $t' = M_{-\rho}(t)$, so $\Phi_\rho$ is this dilation.
\end{proof}

\begin{proof}[Proof of Theorem~\ref{thm:transport}]
Let $P_0$ be the uniform law and $T_{\#}P\defin P\circ T^{-1}$ denote the pushforward of a probability measure $P$ under a measurable map $T$. Then $\rd P_{\bX}/\rd P_0=\omega_d f$ and $\mathcal{M}_{G\#}P_0=G$, so bijectivity gives $(\rd P_{\bY}/\rd G)(\by)=\omega_d f\{\mathcal{M}_G^{-1}(\by)\}$. Multiplying by the density of $G$ with respect to the surface measure $\sigma_d$, $\rd G/\rd\sigma_d=h_G(\by^\top\bmu)/\omega_d$, gives \eqref{eq:pushforward}.
\end{proof}

\begin{proof}[Proof of Proposition~\ref{prop:proj}]
From the tangent--normal factorization, if $\bX$ has the rotationally symmetric density $\bx\mapsto c\,g(\bx^\top\bmu)$ on $\Sp^d$, then $\bX^\top\bmu$ has density $t\mapsto c\,\omega_{d-1}\,g(t)(1-t^2)^{d/2-1}$ on $[-1,1]$. For $d=2$, the projected densities are $t\mapsto(1-r^2)(1+r^2-2rt)^{-3/2}/2$ and $t\mapsto\{1+\gamma P_k(t)\}/2$. Integrating the first via $u=1+r^2-2rt$ gives \eqref{eq:projP}; \eqref{eq:projcard} follows from $\int_{-1}^x P_k(t)\,\rd t=\{P_{k+1}(x)-P_{k-1}(x)\}/{(2k+1)}$.
\end{proof}

\begin{proof}[Proof of Corollary~\ref{cor:quantile}]
The Poisson and order-$1$ cardioid formulas solve \eqref{eq:projP} and the quadratic \eqref{eq:projcard}. For order $2$, $F_{\Card_2}(x)=u$ is the depressed cubic $x^3+px+q=0$; as $p>0$, Cardano's real root gives \eqref{eq:Q2}.
\end{proof}

\begin{proof}[Proof of Proposition~\ref{prop:svmf}]
Write $\bx=t\bmu_2+(1-t^2)^{1/2}\,\bB_{\bmu_2}\bxi$. Then $\boldsymbol{A}\bx=t\bmu_2+a(1-t^2)^{1/2}\,\bB_{\bmu_2}\bxi$ and $\|\boldsymbol{A}\bx\|=\{a^2+(1-a^2)t^2\}^{1/2}$. Hence, with $\tau_a(t)=t/\{a^2+(1-a^2)t^2\}^{1/2}$,
\begin{align*}
  S_a(\bx)=\tau_a(t)\bmu_2+\{1-\tau_a(t)^2\}^{1/2}\,\bB_{\bmu_2}\bxi,
\end{align*}
where $1-\tau_a(t)^2=a^2(1-t^2)\{a^2+(1-a^2)t^2\}^{-1}$. Thus $S_a$ rearranges the cosine by $\tau_a$ and fixes $\bxi$. Since $G_a$ is the law of $S_a(\bU)$ for a uniform $\bU$, the definition of $\mathcal{M}_{G_a}$ gives $\mathcal{M}_{G_a}=S_a$, and $G_a$ has the angular central Gaussian density \eqref{eq:acg} with parameter $\boldsymbol{A}^2$. Direct computation gives $\tau_a'(t)=a^2/\{a^2+(1-a^2)t^2\}^{3/2}>0$ and $\tau_a^{-1}=\tau_{1/a}$. Finally, \eqref{eq:svmf} is \eqref{eq:genvmf} with $\bmu_1=\bmu_2=\bmu$ and $h_{G_a}(s)=a\{1-(1-a^2)s^2\}^{-(d+1)/2}$, using $\bmu^\top\mathcal{M}_{G_a}^{-1}(\by)=\tau_a^{-1}(\by^\top\bmu)=a(\by^\top\bmu)\{1-(1-a^2)(\by^\top\bmu)^2\}^{-1/2}$.
\end{proof}


\fi

\end{document}